\documentclass[11pt]{article}
\usepackage{indentfirst}
\usepackage{amsmath, amscd, amssymb, amsthm, verbatim}

\usepackage{mathrsfs}

\allowdisplaybreaks

\newtheorem{theorem}{Theorem}[section]

\newtheorem{lemma}[theorem]{Lemma}

\begin{document}
\title{On an
area-preserving inverse curvature flow for plane curves}

\author{
  Zezhen Sun
    \thanks{Corresponding author. The Department of General Education, Anhui Water Conservancy Technical College, Hefei 231603, China. E-mail address: \texttt{52205500017@stu.ecnu.edu.cn} }
  \and
  Yuting Wu
    \thanks{College of Mathematics and Statistics, Northwest Normal University, Lanzhou 730070, China. E-mail address: \texttt{52215500001@stu.ecnu.edu.cn}}
}

\maketitle
\begin{abstract}
  \noindent
  In this paper, we study a $1/\kappa^{n}$-type area-preserving non-local flow of convex closed plane curves for any $n>0$. We show that the flow exists globally, the length of evolving curve is
 non-increasing, and the limiting curve will be a circle in the $C^{\infty}$ metric as time $t\to\infty$.
  \\
  \\
  {\bf Keywords:}Area-preserving, Inverse curvature flow, Existence, Convergence
  \\
  {\bf Mathematics Subject Classification:} 53E99,35B40,35K55,
\end{abstract}

\section{Introduction}
Over the past several decades, geometric analysts have made extensive investigations into the problem of curve evolution within the planar domain. The most famous one is curve shortening flow studied by Gage–Hamilton \cite{Gage-Hamilton,Gage83,Gage84},  Grayson \cite{Grayson87}, Abresch–Langer \cite{uajl}, and others. A wealth of profound findings concerning the evolution of curves driven by curvature can be found in Chou and Zhu’s monograph \cite{Chou-Zhu} and the associated references therein. Additionally, due to applications in fields such as fluid dynamics, phase transitions, and image processing, people began to study the non-local curve flows, such as length-preserving flows \cite{gpt1,mz12,Tsai-Wang,p01,Pan-Yang}, area-preserving flows\cite{gpt2,Gage85,Ma-Cheng,Mao-Pan-Wang,Sun}.

In \cite{Lin-Tsai2012}, Lin-Tsai summarized previous nonlocal curve flows to the following general form
\begin{equation}\label{fl}
  \left\{\begin{aligned}
    \frac{\partial X(u, t)}{\partial t}&=\big[F(\kappa(u,t))-\lambda(t)\big]\mathbf{N}(u,t),\\
    X(u, 0)&=X_{0}(u),
  \end{aligned}
  \right.
  \end{equation}
where $X_{0}(u)\subset\mathbb{R}^{2}$ is a given smooth closed curve, parameterized by $u\in S^{1}$, and $X(u, t):S^{1}\times[0,T)\rightarrow\mathbb{R}^{2}$ is a family of curves moving along its inward normal direction $
\mathbf{N}(u,t)$ with given speed function $F(\kappa(u,t))-\lambda(t)$. Here $F(\kappa(u,t))$ is a given function of the curvature $\kappa$ satisfying the parabolic condition $F^{'}(z)>0$
for all $z$ in its domain. The term $\lambda(t)$  is a time-dependent function determined by global geometric quantities associated with the evolving curve $X (\cdot, t)$, such as its length $L(t)$, enclosed area $A(t)$, or integrals of curvature over the curve.

In \cite{yyl}, Yang-Zhao-Zhang studied an area-preserving inverse curvature flow of the form
\begin{equation}\label{fl0}
  \left\{\begin{aligned}
    \frac{\partial X(u, t)}{\partial t}&=\big(p\lambda(t)-\frac{1}{\kappa}\big)\mathbf{N}(u,t),\\
    X(u, 0)&=X_{0}(u),\quad u\in S^{1},
  \end{aligned}
  \right.
  \end{equation}
where $p=-<X,\mathbf{N}>$ is the support function and the non-local term $\lambda(t)=\frac{1}{2A}\int_{X (\cdot, t)}\kappa^{-1}ds$. Here $s$ is the arc length parameter of $X (\cdot, t)$ and $A$ is the enclosed area. It is evident that flow \eqref{fl0} differs structurally from previous non-local flow models. Motivated by their work, we consider the following
 $1/\kappa^{n}$-type nonlocal area-preserving curvature flow
\begin{equation}\label{fl}
  \left\{\begin{aligned}
    \frac{\partial X(u, t)}{\partial t}&=\big(p\lambda(t)-\frac{1}{\kappa^{n}}\big)\mathbf{N}(u,t),\\
    X(u, 0)&=X_{0}(u),\quad u\in S^{1},
  \end{aligned}
  \right.
  \end{equation}
where $\lambda(t)=\frac{1}{2A}\int_{X (\cdot, t)}\kappa^{-n}ds$ and $n>0$ is a constant. It can be seen that when $n=1$, the flow \eqref{fl} reduces to the flow \eqref{fl0}. Therefore, our flow \eqref{fl} can be regarded as a generalization of flow \eqref{fl0}. Besides, under the flow \eqref{fl}, the singularity does not occur (i.e.,curvature of the evolving curve does not blow up to $+\infty$). This behavior differs from previous $1/\kappa^{n}$-type inverse curvature flows \cite{gpt1,gpt2}. The main result of this paper is as follows.

\begin{theorem}\label{them}
A strictly convex closed  plane curve $X_{0}(u)$ which evolves according to \eqref{fl} remains convex,
preserves its enclosed area and decreases in length, becomes more and more circular during the evolution process. Moreover, the limiting curve will be a finite circle with radius $\sqrt{A(0)/\pi}$ in the $C^{\infty}$ metric as $t\to\infty$.
\end{theorem}
In what follows, we use subscripts to denote partial derivatives.
\section{Proof of Theorem \ref{them}}
The proof of Theorem \ref{them} will be broken down into a number of lemmas. Since the case where \(n = 1\) has already been studied in \cite{yyl}, this paper will only consider the cases where \(0 < n < 1\) and \(n > 1\).
In what follows, we use subscripts to denote partial derivatives.
\subsection{Preparations}
Let the metric along the curve be defined as $g(u,t)=|X_{u}|=(x^{2}_{u}+y^{2}_{u})^{\frac{1}{2}}$. The arc-length element is then expressed as $ds=g(u,t)du$, or formally
$$\frac{\partial}{\partial s}=\frac{1}{g}\frac{\partial}{\partial u},\quad \frac{\partial s}{\partial u}=g.$$
The tangent $\mathbf{T}$, normal $\mathbf{N}$, tangent angle $\theta$, curvature $\kappa$ and the length $L$ of the curve and the area $A$ it bounds are defined in the standard way:
\begin{align*}
\mathbf{T}&=\frac{\partial X}{\partial s}=\frac{1}{g}\frac{\partial X}{\partial u},\quad \kappa=\frac{\partial\theta}{\partial s},\quad \mathbf{N}=\frac{1}{\kappa}\frac{\partial T}{\partial s}=\frac{1}{\kappa g}\frac{\partial T}{\partial u},\\
L(t)&=\int^{b}_{a}g(u,t)du=\oint ds,\quad A(t)=\frac{1}{2}\oint xdy-ydx=-\frac{1}{2}\oint<X,\mathbf{N}>ds.
\end{align*}
Since changing the tangential components of $\frac{\partial X}{\partial t}$ influences only the parametrization, not the geometric shape of the evolving curve, we are free to select an appropriate tangential component
 $\alpha=-\frac{\partial}{\partial\theta}(p\lambda(t)-\kappa^{-n})$. This choice enables the simplification of the geometric analysis of the evolving curves, that is, we can instead consider the following equivalent evolution problem
\begin{equation}\label{fl2}
  \left\{\begin{aligned}
    \frac{\partial X(u, t)}{\partial t}&=-\frac{\partial}{\partial\theta}\big(p\lambda(t)-\kappa^{-n}\big)\mathbf{T}(\theta,t)
    +\big(p\lambda(t)-\kappa^{-n}\big)\mathbf{N}(\theta,t),\\
    X(\theta, 0)&=X_{0}(\theta),\quad (\theta,t)\in S^{1},
  \end{aligned}
  \right.
  \end{equation}

The following evolution equations are standard and can be obtained following the approach in \cite{Chou-Zhu}.
\begin{lemma}\label{eeg}
Under the flow \eqref{fl2}, the geometric quantities for the length, the enclosed area,  the curvature and the support function evolve as
\begin{align}
\label{at}\frac{dA}{dt}&=-\int_{0}^{L}\big(p\lambda(t)-\kappa^{-n}\big)ds
=\int_{0}^{2\pi}\kappa^{-n-1}d\theta-2A\lambda(t),\\
\label{lt}\frac{dL}{dt}&=-\int_{0}^{L}\kappa\big(p\lambda(t)-\kappa^{-n}\big)ds
=\int_{0}^{2\pi}\kappa^{-n}d\theta-L\lambda(t),\\
\label{kt}\frac{\partial\kappa}{\partial t}&=\kappa^{2}\bigg[\big(p\lambda(t)-\kappa^{-n}\big)_{\theta\theta}+
p\lambda(t)-\kappa^{-n}\bigg]=\kappa^{2}\big[\kappa^{-1}\lambda(t)-\kappa^{-n}-(\kappa^{-n})_{\theta\theta}\big],
\end{align}
where the facts that $\int_{0}^{L}\kappa pds=L, \int_{0}^{L} pds=2A$ and $p_{\theta\theta}+p=\frac{1}{\kappa}$ have been used.
\end{lemma}

\begin{lemma}\label{llaa}
(The monotonicity of length and area)Under the flow \eqref{fl2}, we have
\begin{equation*}
\frac{dA(t)}{dt}=0,
\quad\frac{dL(t)}{dt}\le0\quad\forall t\in[0,T).
\end{equation*}
\end{lemma}
\begin{proof}
By \eqref{at}, one can easily see that $dA/dt=0$. As for the length, by the classical isoperimetric inequality
$L^2-4\pi A\ge0$ and the H$\ddot{o}$lder inequality
\begin{equation}\label{hoi1}
\int^{2\pi}_{0}\frac{1}{\kappa}d\theta\int^{2\pi}_{0}\frac{1}{\kappa^{n}}d\theta\le\int^{2\pi}_{0}d\theta
\int^{2\pi}_{0}\frac{1}{\kappa^{n+1}}d\theta,
\end{equation}
we have
$$\int_{0}^{2\pi}\frac{1}{\kappa^{n}}d\theta\le\frac{2\pi}{L}\int_{0}^{2\pi}\frac{1}{\kappa^{n+1}}d\theta
\le\frac{L}{2A}\int_{0}^{2\pi}\frac{1}{\kappa^{n+1}}d\theta,$$
which, together with \eqref{lt} yields $dL/dt\le0$.
\end{proof}
As a direct consequence of Lemma \eqref{llaa}, we can easily get
\begin{lemma}\label{lajie}
Under the flow \eqref{fl2}, we have
\begin{equation}\label{la}
 A(0)= A(t)\quad \text{and} \quad\sqrt{4\pi A(0)}\le L(t)\le L(0),\quad\forall t\in[0,T),
\end{equation}
and the isoperimetric ratio $\frac{L^{2}(t)}{4\pi A(t)}$ is decreasing in time $t\in[0,T)$.
\end{lemma}

Since the radius of curvature $\rho(\theta,t)=\frac{1}{\kappa(\theta,t)}$, we can get the following evolution equation
\begin{equation}\label{rt}
\frac{\partial\rho}{\partial t}=(\rho^{n})_{\theta\theta}+\rho^{n}-
\frac{\rho}{2A}\int^{2\pi}_{0}\rho^{n+1}d\theta,\quad(\theta,t)\in S^{1}\times[0,T)
\end{equation}
with $\rho(\theta,0)=\rho_{0}(\theta)=\frac{1}{\kappa_{0}(\theta)}>0$.  Set $\varphi(\theta,t)=\rho^{n}(\theta,t)$, then \eqref{rt} can be expressed as
\begin{equation}\label{nt}
\frac{\partial\varphi}{\partial t}=n\varphi^{1-\frac{1}{n}}[\varphi_{\theta\theta}+\varphi-\lambda(t)\varphi^{\frac{1}{n}}],\quad
\lambda(t)=\frac{1}{2A}\int^{2\pi}_{0}\varphi^{1+\frac{1}{n}}d\theta.
\end{equation}
The asymptotic behavior of the flow solution $X(\cdot,t)$ is determined by the asymptotic behavior of the scalar solution \eqref{kt}, or \eqref{rt},or \eqref{nt}.

\begin{lemma}\label{lamj}
(Bounds on the nonlocal term $\lambda(t)$)
Under the flow \eqref{fl2}, there exist two positive constants $M_{1}$ and $M_{2}$, independent of time,such that
\begin{equation}\label{laj}
M_{1}\le\lambda(t)\le M_{2},\quad\forall t\in[0,T).
\end{equation}
\end{lemma}
\begin{proof}
For the lower bound, by the H$\ddot{o}$lder inequality
\begin{equation}\label{hold}
L(t)=\int^{2\pi}_{0}\frac{1}{\kappa(\theta,t)}d\theta\le\bigg(\int^{2\pi}_{0}\frac{1}{\kappa^{n+1}(\theta,t)}d\theta\bigg)
^{\frac{1}{n+1}}\bigg(\int^{2\pi}_{0}d\theta\bigg)^{\frac{n}{n+1}},
\end{equation}
we have
$$\lambda(t)=\frac{1}{2A}\int^{2\pi}_{0}\frac{1}{\kappa^{n+1}}d\theta\ge\frac{L^{n+1}(t)}{2A(2\pi)^{n}}\ge
\frac{(4\pi A(0))^{\frac{n+1}{2}}}{2A(0)(2\pi)^{n}}\triangleq M_{1}.$$
 For the upper bound, we calculate
\begin{align}\label{la1}
\frac{d\lambda(t)}{dt}&=\frac{d}{dt}\bigg(\frac{1}{2A(0)}\int^{2\pi}_{0}\varphi^{1+\frac{1}{n}}d\theta\bigg)=
\frac{n+1}{2nA(0)}\int^{2\pi}_{0}\varphi^{\frac{1}{n}}\varphi_{t}d\theta\notag\\
&=\frac{n+1}{2nA(0)}\int^{2\pi}_{0}\varphi^{\frac{1}{n}}
\bigg(n\varphi^{1-\frac{1}{n}}[\varphi_{\theta\theta}+\varphi-\lambda(t)\varphi^{\frac{1}{n}}]\bigg)d\theta\notag\\
&=\frac{n+1}{2A(0)}\bigg[\int^{2\pi}_{0}-(\varphi_{\theta})^{2}d\theta+\int^{2\pi}_{0}\varphi^{2}d\theta
-\frac{1}{2A(0)}\bigg(\int^{2\pi}_{0}\varphi^{1+\frac{1}{n}}d\theta\bigg)^{2}\bigg].
\end{align}
Set $\bar{\varphi}=\frac{1}{2\pi}\int^{2\pi}_{0}\varphi d\theta.$ Applying Wirtinger's inequality yields
\begin{equation}\label{la11}
\int^{2\pi}_{0}(\varphi_{\theta})^{2}d\theta\ge\int^{2\pi}_{0}(\varphi-\bar{\varphi})^{2}d\theta=\int^{2\pi}_{0}\varphi^{2}d\theta-2\pi \bar{\varphi}^{2}.
\end{equation}
This implies
\begin{equation}\label{la2}
\int^{2\pi}_{0}\varphi^{2}d\theta\le\int^{2\pi}_{0}(\varphi_{\theta})^{2}d\theta+\frac{1}{2\pi}\big(\int^{2\pi}_{0}\varphi d\theta\big)^2.
\end{equation}
By isoperimetric inequality $L^{2}(t)\ge4\pi A(t)$ and Lin-Tsai inequality \cite{Lin-Tsai2012}
$$\int^{2\pi}_{0}\rho^{\alpha}d\theta\le\frac{2\pi}{L}\int^{2\pi}_{0}\rho^{\alpha+1}d\theta,\quad \forall\alpha\ge0,$$
we obtain
\begin{equation}\label{la3}
\int^{2\pi}_{0}\varphi^{1+\frac{1}{n}}d\theta\ge\frac{L}{2\pi}\int^{2\pi}_{0}\varphi d\theta\ge\sqrt{\frac{A}{\pi}}\int^{2\pi}_{0}\varphi d\theta.
\end{equation}
By combining \eqref{la1} and \eqref{la2}, we conclude that $d\lambda(t)/dt\le0$, meaning $\lambda(t)$ is decreasing in time. Thus, $\lambda(t)\le\lambda(0)$. The proof is done.
\end{proof}

\begin{lemma}\label{est}
Under the flow \eqref{fl2}, there holds the estimate
\begin{equation}
\max\limits_{S^{1}\times[0,t]}\Phi\le\max\bigg\{\max\limits_{S^{1}\times[0,t]}\varphi^{2},\max\limits_{S^{1}\times\{0\}}\Phi\bigg\},
\end{equation}
where $\Phi=\varphi^{2}+(\varphi_{\theta})^{2}$.
\end{lemma}
\begin{proof}
Consider an arbitrary time $t > 0$ within the solution's existence interval $[0, T)$.
Assume $\Phi$ attains its maximum value over the domain $S^1\times [0, t]$ at some point $(\theta_0, t_0)$.
If $t_0 = 0$, the result holds trivially, so we focus on the case where $t_0 > 0$.

We now claim that $\varphi_{\theta}(\theta_0,t_0)=0$ and hence the conclusion stands proven.
Suppose to the contrary that $\varphi_\theta(\theta_0, t_0) \neq 0$.
At this maximum point,  we have $\varphi_\theta(\varphi_{\theta\theta} + \varphi) = 0$,
which would imply $\varphi_{\theta\theta} + \varphi = 0$. A direct calculation shows that at $(\theta_{0},t_{0})$, we have
\[
\frac{\partial \Phi}{\partial t} = n\varphi^{1-\frac{1}{n}}\Phi_{\theta\theta} - 2n\lambda(t)\varphi^{2} - 2n\lambda(t)(\varphi_{\theta})^{2}< 0,
\]
which contradicts the requirement that $\partial \Phi / \partial t \geq 0$ at a point where $\Phi$ achieves its maximum. This completes the proof.
\end{proof}

\begin{lemma}\label{new}
Under the flow \eqref{fl2}, assume that at some point $(\theta_{0},t_{0})\in S^{1}\times(0,T)$, we have
\begin{equation*}
\varphi(\theta_{0},t_{0})=\max\limits_{S^{1}\times[0,t_{0}]}\varphi(\theta,t).
\end{equation*}
Then for any small $\epsilon>0$, there exists $\eta>0$ depending only on $\epsilon$ such that
\begin{equation*}
(1-\epsilon)\varphi(\theta_{0},t_{0})\le\varphi(\theta,t_{0})+\epsilon\sqrt{c}
\end{equation*}
holds for all $\forall\theta\in(\theta_{0}-\eta,\theta_{0}+\eta)$, where $c$ is a constant depending solely on the initial curve.
\end{lemma}
\begin{proof}
By Lemma \ref{est},we have
\begin{align*}
\varphi(\theta_{0},t_{0})&=\varphi(\theta,t_{0})+\int^{\theta_{0}}_{\theta}\varphi_{\theta}(\theta,t_{0})d\theta\\
&\le\varphi(\theta,t_{0})+\lvert\theta_{0}-\theta\lvert\max\limits_{\theta\in S^{1}}\lvert\varphi_{\theta}(\theta,t_{0})\lvert\\
&\le\varphi(\theta,t_{0})+\lvert\theta_{0}-\theta\lvert\sqrt{\max\limits_{S^{1}\times[0,t_{0}]}\varphi^{2}(\theta,t)+c}\\
&=\varphi(\theta,t_{0})+\lvert\theta_{0}-\theta\lvert\sqrt{\varphi^{2}(\theta_{0},t_{0})+c}\\
&\le\varphi(\theta,t_{0})+\eta\varphi(\theta_{0},t_{0})+\eta\sqrt{c},
\end{align*}
where the last inequality uses $|\theta - \theta_0| < \eta$. Rearranging terms yields
\begin{equation*}
(1 - \eta)\varphi(\theta_0, t_0) \leq \varphi(\theta, t_0) + \eta \sqrt{c}.
\end{equation*}
The result follows by setting $\eta=\epsilon$.
\end{proof}
\begin{lemma}\label{lok}(Lower bound of the curvature)
Under the flow \eqref{fl2}, there exists a constant $M_{3}>0$ ,which is independent of time, such that
\begin{equation*}
\kappa(\theta,t)\ge M_{3}>0,\quad\forall(\theta,t)\in S^{1}\times[0,T).
\end{equation*}
Therefore, the evolving curve is uniformly convex on $[0,T).$
\end{lemma}
\begin{proof}
First, we assert that there exists a constant $M_{4}>0$, independentof time, such that
\begin{equation}\label{nuu}
\max\limits_{S^{1}\times[0,T)}\varphi(\theta,t)\le M_{4}.
\end{equation}
Suppose, for contradiction, that \eqref{nuu} does not hold, then we can find a sequence $\{t_{i}\}^{\infty}_{i=1}\rightarrow T$ such that $\max_{S^{1}}\varphi(\theta,t_{i})\to\infty$ as $t\rightarrow\infty$. By Lemma \ref{new}, one has $L(t_{i})=\int^{2\pi}_{0}\varphi^{\frac{1}{n}}(\theta,t_{i})d\theta\to\infty$ as $t\rightarrow\infty$. This stands in contradiction that $L(t)\le L(0)$, so our initial assertion is valid. Then by \eqref{nuu}, it follows that
\begin{equation}\label{ke}
\bigg(\frac{1}{\min\limits_{S^{1}\times[0,T)}\kappa(\theta,t)}\bigg)^{n}=\max\limits_{S^{1}\times[0,T)}\omega(\theta,t)\le M_{4},
\end{equation}
which shows that the curvature $\kappa(\theta,t)$ has a time-independent positive lower bound. This completes the proof.
\end{proof}
By mimicking the proof in \cite{Mao-Pan-Wang}, we can establish the short time existence of flow \eqref{fl2}.
\begin{theorem}
There exists $T > 0$ such that the flow \ref{fl2} admits a unique solution $X(\theta,t) \in S^1 \times [0,T]$, where $T$ is a small positive time.
\end{theorem}
\section{Long time Existence }
In this section, we will use the maximum principle to prove the long time existence of the flow \eqref{fl2}.
\begin{theorem}(Long-time existence)
The evolution equation \eqref{fl2} has a long time solution on $S^1\times[0,\infty)$.
\end{theorem}

\begin{proof}
Let $\theta_{\ast}$ be the point such that $\kappa(\cdot,t)$ attains its maximum value, that is
$$\kappa(\theta_{\ast},t)=\max\{\kappa(\theta,t)|\theta\in S^1\}\triangleq\kappa_{\max}(t).$$
At the point $(\theta_{\ast},t)$, we have
\begin{align*}
\frac{\partial \kappa}{\partial t} &= \kappa^2[\kappa^{-1}\lambda(t)-\kappa^{-n}-(\kappa^{-n})_{\theta\theta}]\\
&=\kappa\lambda(t)-\kappa^{2-n}+n\kappa^{1-n}\kappa_{\theta\theta}-n(n+1)\kappa^{-n}(\kappa_{\theta})^2\\
&\le\kappa\lambda(t)-\kappa^{2-n}\le  M_2 \kappa,
\end{align*}
where $M_2$ is the upper bound of $\lambda(t)$(see Lemma \ref{lamj}). So we obtain
$$\frac{d\kappa_{\text{max}}}{dt} \leq M_2 \kappa_{\text{max}}.$$
Solving this differential inequality yields
\begin{equation}\label{ks0}
\kappa_{\text{max}}(t) \leq \kappa_{\text{max}}(0)  e^{M_2 t}.
\end{equation}
Therefore, the singularity will never happen as time goes. This completes the proof.
\end{proof}
Although long-time existence for flow \eqref{fl2} on
$S^1\times[0,\infty)$ follows from \eqref{ks0}, we can obtain a better upper bound estimate for the curvature.
\begin{lemma}\label{rhodd}
Under the flow \eqref{fl2}, there exists $\theta(t)\in S^1$ such that
\begin{equation}\label{rhod}
\rho(\theta(t),t)=\sqrt{\frac{A(0)}{\pi}},\quad \forall t\in[0,\infty).
\end{equation}
\end{lemma}
\begin{proof}
Set $\min\{\rho(\theta,t)|\theta\in S^1\}\triangleq\rho_{\min}(t)$ and $\max\{\rho(\theta,t)|\theta\in S^1\}\triangleq\rho_{\max}(t).$ By the isoperimetric inequality $\int^{2\pi}_{0}\rho d\theta=L(t)\ge\sqrt{4\pi A(0)}$, one gets
\[\rho_{\max}(t)\ge\frac{1}{2\pi}\int^{2\pi}_{0}\rho d\theta=\frac{L(t)}{2\pi}\ge\sqrt{\frac{A(0)}{\pi}}.\]
On the other hand, from the Gage inequality, it follows that
\[\frac{2\pi}{\rho_{\min}(t)}\ge\int^{2\pi}_{0}\frac{1}{\rho(\theta,t)} d\theta=\int^{2\pi}_{0}\kappa(\theta,t) d\theta\ge\frac{\pi L(t)}{A(0)}\ge\frac{\pi\sqrt{4\pi A(0)}}{A(0)}.\]
Thus,
\[\rho_{\min}(t)\le\sqrt{\frac{A(0)}{\pi}}\le\rho_{\max}(t).\]
By the continuity of $\rho$, we can obtain \eqref{rhod}.
\end{proof}

\begin{lemma}\label{kupj}
There exists a positive constant $J$, depending only on $n$ and $X_0$, such that
\begin{equation}
\kappa\le J,\quad \forall (\theta,t)\in S^1\times[0,\infty).
\end{equation}
\end{lemma}
\begin{proof}
Lemma \ref{lok} guarantees a uniform lower bound for the curvature, i.e.,$\kappa(\theta,t)\ge M_3>0$, which implies $\rho(\theta,t)\le J_1\triangleq1/M_3$. Define the energy functional
\begin{equation*}
E(t)=\frac{1}{2}\int^{2\pi}_{0}\rho^2d\theta.
\end{equation*}
Differentiating with respect to time and using the evolution equation of $\rho$ yields
\begin{align*}
\frac{dE}{dt}=\int^{2\pi}_{0}\rho\rho_td\theta&=\int^{2\pi}_{0}\rho[(\rho^{n})_{\theta\theta}+\rho^n-\lambda(t)\rho]d\theta\\
&=-n\int^{2\pi}_{0}\rho^{n-1}\rho^{2}_{\theta}d\theta+\int^{2\pi}_{0}\rho^{n+1}d\theta-\lambda(t)\int^{2\pi}_{0}\rho^{2}d\theta.
\end{align*}
Since $\rho\le J_1$ and $\lambda(t)\le M_2$, we have the estimates
\begin{equation*}
\bigg|\int^{2\pi}_{0}\rho^{n+1}d\theta\bigg|\le2\pi J_1^{n+1},\quad\bigg|\lambda(t)\int^{2\pi}_{0}\rho^{2}d\theta\bigg|\le2\pi M_2J_1^2.
\end{equation*}
Set $J_2 = 2\pi J_1^{n+1} + 2\pi M_2J_1^2$. Then we have
\begin{equation}\label{eti}
\frac{dE}{dt} \leq -n f(t) + J_2.
\end{equation}
where $f(t) = \int_0^{2\pi} \rho^{n-1} \rho_\theta^2  d\theta \geq 0.$
The isoperimetric inequality $L^2 \geq 4\pi A$ and Cauchy-Schwarz imply
\[
\int_0^{2\pi} \rho^2  d\theta \geq \frac{L^2}{2\pi} \geq 2A,
\]
so $E(t) \geq A(0) > 0$. Integrating the differential inequality \eqref{eti} yields
\[
E(t) - E(0) \leq \int_0^t \left( -n f(z) + J_2 \right) dz,
\]
If $\int_0^\infty f(t)  dt = \infty$, then $E(t) - E(0) \leq -n \int_0^t f(z)  dz + J_2 t \to -\infty$ as $t \to \infty$, contradicting $E(t) \geq A(0)$. Hence
\[
\int_0^\infty f(t)  dt < \infty.
\]

Let $r_0 = \sqrt{A(0)/\pi}$ and choose $b = r_0/(8M_2 J_1)$. For any given $\epsilon>0$, there exists an integer $I_0 > 0$ such that for all $i \geq I_0$:
\[
 f(t_i) < \epsilon,\quad t_i \in [ib, (i+1)b].
\]
where the number $\epsilon$ will be chosen appropriately later on.
By Lemma \ref{rhodd},  there exists $\theta_i \in S^1$ such that $\rho(\theta_i, t_i) = r_0$. Let $\theta_* \in S^1$ satisfy $\rho(\theta_*, t_i) = \rho_{\min}(t_i)$. By the H$\ddot{o}$lder inequality, we have
\begin{align}
\rho^{\frac{n+1}{2}}_{\min}(t_i) - r_0^{\frac{n+1}{2}}&=\rho^{\frac{n+1}{2}}(\theta_{\ast},t_i)-\rho^{\frac{n+1}{2}}(\theta_{i},t_i)\notag\\
&=\int^{\theta_{\ast}}_{\theta_{i}}\frac{n+1}{2}\rho^{\frac{n-1}{2}}(\theta,t_i)\rho_{\theta}(\theta,t_i)d\theta\notag\\
&\ge-\frac{n+1}{2}\sqrt{2\pi f(t_i)}.
\end{align}
So,
\[\rho^{\frac{n+1}{2}}_{\min}(t_i)\ge r_0^{\frac{n+1}{2}}-\frac{(n+1)\sqrt{2\pi f(t_i)}}{2}\ge
r_0^{\frac{n+1}{2}}-\frac{(n+1)\sqrt{2\pi\epsilon}}{2}.\]
If we select a sufficiently small $\epsilon>0$, we will have
\[\rho_{\min}(t_i)\ge \frac{r_{0}}{2}.\]
For $t \in [(I_0+1)b, \infty)$, select $t_i \in [ib, (i+1)b]$ ($i \geq I_0$) such that $|t - t_i| \leq 2b$. Let $\theta(t) \in S^1$ be such that $\rho(\theta(t), t) = \rho_{\min}(t)$. At $\theta(t)$, the evolution equation (2.16) gives
\[
\frac{\partial \rho}{\partial t}(\theta(t), t) \geq\rho^{n}(\theta(t), t) -\lambda(t) \rho\ge-\lambda(t) \rho\ge -M_2 J_1,
\]
 Integrating over $[t_i, t]$ yields
\[
\rho_{\min}(t) - \rho_{\min}(t_i) \geq -M_2 J_1 (t - t_i) \geq -2M_2 J_1 b.
\]
Thus:
\[
\rho_{\min}(t) \geq \rho_{\min}(t_i) - 2M_2 J_1 b \geq r_0/2 - 2M_2 J_1 \cdot \frac{r_0}{8M_2 J_1} = \frac{r_0}{4}.
\]
On $[0, (I_0+1)b]$, the continuity and positivity of $\rho$ imply $\rho \geq c_0 > 0$ for some constant $c_0$. Thus, there exists $J^{-1}\triangleq\min\{c_0, r_0/4\}$ such that $\rho(\theta, t) \geq J^{-1}$ for all $(\theta, t)$, and consequently
\[
\kappa(\theta, t) = 1/\rho(\theta, t) \leq J.
\]
\end{proof}

\section{Convergence of the flow}
\begin{lemma}\label{isod}
(Convergence of the isoperimetric difference)
Under the flow \eqref{fl2}, the isoperimetric difference $L^{2}(t)-4\pi A(t)$ decreases and decays to zero exponentially as $t$ approaches infinity.
\end{lemma}
\begin{proof}
It follows easily from Lemma \eqref{llaa} that $d(L^{2}(t)-4\pi A(t))/dt\le0,$ therefore, $L^{2}(t)-4\pi A(t)$ decreases over time. By \eqref{hoi1} and \eqref{hold}, we have
\begin{align*}
\frac{d}{dt}(L^{2}(t)-4\pi A(t))&=2L(t)\int^{2\pi}_{0}\frac{1}{\kappa^{n}}d\theta-\frac{L^{2}(t)}{A(t)}
\int^{2\pi}_{0}\frac{1}{\kappa^{n+1}}d\theta\\
&\le-\frac{L^{2}(t)-4\pi A(t)}{A(t)}\int^{2\pi}_{0}\frac{1}{\kappa^{n+1}}d\theta\\
&\le-\frac{L^{n+1}(t)}{(2\pi)^{n}A(t)}\big(L^{2}(t)-4\pi A(t)\big)\\
&\le-\frac{(4\pi A(0))^{\frac{n+1}{2}}}{(2\pi)^{n}A(0)}\big(L^{2}(t)-4\pi A(t)\big)\\
&=-2\bigg(\frac{A(0)}{\pi}\bigg)^{\frac{n-1}{2}}\big(L^{2}(t)-4\pi A(t)\big).
\end{align*}
Integrating this yields
\begin{equation}
L^{2}(t)-4\pi A(t)\le \big(L^{2}(0)-4\pi A(0)\big)e^{-2\big(\frac{A(0)}{\pi}\big)^{\frac{n-1}{2}}t}.
\end{equation}
This completes the proof.
\end{proof}
Applying Lemma \ref{isod} and the Bonnesen inequality establishes the following result easily, we omit the details of the proof.
\begin{lemma}
Under the flow \eqref{fl2}, the evolving curve converges to a finite circle in the Hausdorff metric.
\end{lemma}
\begin{lemma}\label{ntjj}
(Uniform boundedness of spatial derivatives)
Under the flow \eqref{fl2}, there exist positive constants $D_{i}$
(i=1,2,3\dots), independent of time, such that
\begin{equation}\label{ntj}
\bigg\lvert\frac{\partial^{i}\varphi}{\partial\theta^{i}}(\theta,t)\bigg\lvert\le D_{i},\quad\forall (\theta,t)\in S^{1}\times[0,\infty).
\end{equation}
\end{lemma}

\begin{proof}
The base case $k=1$ is treated in detail, higher derivatives follow similarly by induction. By Lemma \ref{lok} and Lemma \ref{kupj}, there exist two positive constants $m$ and $M$ such that
\begin{equation}\label{nuji}
0<m\le\varphi(\theta,t)\le M.
\end{equation}
Define the auxiliary function
\[
\Psi(\theta,t) \triangleq \varphi_\theta + \mu \varphi
\]
where $\mu$ is a constant to be determined. By the evolution of $\varphi$, we have
\begin{align}\label{eq:psi_t}
\frac{\partial \Psi}{\partial t} =(\varphi_\theta)_{t} + \mu\varphi_t=&n(1-\frac{1}{n}) \varphi^{-\frac{1}{n}} \varphi_\theta (\varphi_{\theta\theta} + \varphi - \lambda \varphi^{\frac{1}{n}}) + n \varphi^{1 - \frac{1}{n}} (\varphi_{\theta\theta\theta} + \varphi_\theta - \frac{1}{n}\lambda\varphi^{\frac{1}{n}-1} \varphi_\theta)\notag\\
&+n\mu\varphi^{1 - \frac{1}{n}} (\varphi_{\theta\theta} + \varphi - \lambda \varphi^{\frac{1}{n}})\notag\\
=&n \varphi^{1 - \frac{1}{n}}\varphi_{\theta\theta\theta}+\big[(n-1)\varphi^{-\frac{1}{n}}\varphi_\theta
+n\mu\varphi^{1 - \frac{1}{n}}\big]\varphi_{\theta\theta}+\big[(2n-1)\varphi^{1 - \frac{1}{n}}-n\lambda\big]\varphi_\theta\notag\\
&+n\mu\varphi^{2 - \frac{1}{n}}-n\mu\lambda\varphi .
\end{align}
By the identities
\begin{align*}
\varphi_\theta = \Psi -\mu\varphi,\quad
\varphi_{\theta\theta} = \Psi_\theta - \mu(\Psi -\mu\varphi), \quad
\varphi_{\theta\theta\theta} = \Psi_{\theta\theta} - \mu\Psi_\theta +\mu^2(\Psi -\mu\varphi),
\end{align*}
we can rewrite \eqref{eq:psi_t} as
\begin{align*}
\frac{\partial \Psi}{\partial t} =&  n \varphi^{1 - \frac{1}{n}}\big[\Psi_{\theta\theta} - \mu\Psi_\theta +\mu^2(\Psi -\mu\varphi)\big]+\big[(n-1)(\Psi -\mu\varphi)\varphi^{-\frac{1}{n}}
+n\mu\varphi^{1 - \frac{1}{n}}\big]\big[\Psi_\theta - \mu(\Psi -\mu\varphi)\big]\notag\\
&+\big[(2n-1)\varphi^{1 - \frac{1}{n}}-n\lambda\big]( \Psi -\mu\varphi)+n\mu\varphi^{2 - \frac{1}{n}}-n\mu\lambda\varphi\notag\\
=& n \varphi^{1-\frac{1}{n}}\Psi_{\theta\theta}+(n-1)\varphi^{-\frac{1}{n}}( \Psi -\mu\varphi)\Psi_\theta
-(n-1)\mu\varphi^{-\frac{1}{n}}\Psi^2\notag\\
&+\big(\big[2(n-1)\mu^{2}+(2n-1)\big]\varphi^{1 - \frac{1}{n}}-n\lambda\big)\Psi
+\mu(\mu^2+1)(1-n)\varphi^{2 - \frac{1}{n}}
\end{align*}
Notice that $\varphi$ and $\lambda$ are uniformly bounded from above and below. When $n>1$, take $\mu=1$, then $\Psi=\varphi_\theta+\varphi$ and $-(n-1)\mu\varphi^{-\frac{1}{n}}=-(n-1)\varphi^{-\frac{1}{n}}<0$. If $\Psi(\theta_{\ast},t)=\Psi_{\max}(t)$ is large enough (positively) at the maximum point $\theta=\theta_{\ast}$, then we can obtain
\begin{equation*}
\frac{\partial \Psi}{\partial t}\le-(n-1)\varphi^{-\frac{1}{n}}\Psi^2+\big[(4n-3)\varphi^{1 - \frac{1}{n}}-n\lambda\big]\Psi
+2(1-n)\varphi^{2 - \frac{1}{n}}<0\quad \text{at}\quad (\theta_{\ast},t).
\end{equation*}
By the maximum principle, the function $\Psi=\varphi_\theta+\varphi$ is bounded above by a positive constant
that does not depend on time, and this is also true for the function $\varphi_\theta$.
 Similarly, if we take $\mu=-1$, then $\Psi=\varphi_\theta-\varphi$ and $-(n-1)\mu\varphi^{-\frac{1}{n}}=(n-1)\varphi^{-\frac{1}{n}}>0$. If $\Psi(\theta_{\ast},t)=\Psi_{\min}(t)$ is large enough (negatively) at the minimum point $\theta=\theta_{\ast}$, then we have
\begin{equation*}
\frac{\partial \Psi}{\partial t}\ge(n-1)\varphi^{-\frac{1}{n}}\Psi^2+\big[(4n-3)\varphi^{1 - \frac{1}{n}}-n\lambda\big]\Psi
-2(1-n)\varphi^{2 - \frac{1}{n}}>0\quad \text{at}\quad (\theta_{\ast},t).
\end{equation*}
By the minimum principle, the function $\Psi=\varphi_\theta-\varphi$  is bounded below by a negative time-independent constant and so is the function $\varphi_\theta$. When $0<n<1$,  the proof is similar.
\end{proof}

\begin{lemma}\label{lto0}
Under the flow \eqref{fl2}, we have
\begin{equation}
\frac{dL}{dt}\to0\quad\text{as}\quad t\to\infty.
\end{equation}
\end{lemma}
\begin{proof}
We first recall a useful result from calculus. Let $f(t)\ge0$ be a differentiable function on $[0,\infty)$ satisfying $\int^{\infty}_{0}f(t)dt<\infty$. If there exists a positive constant $C$ such that $ f^{'}(t)\le C$ on $[0,\infty)$ or there exists a negative constant $C$ such that $ f^{'}(t)\ge C$ on $[0,\infty)$, then it must be that $f(t)\to0$ as $t\to\infty$.

Let $f(t)=-L^{'}(t).$ By Lemma \ref{llaa}, we know that $f(t)\ge0$ on $[0,\infty)$ with
$$\int^{\infty}_{0}f(t)dt=L(0)-L(\infty)<\infty.$$
Compute
\begin{align*}
\big\lvert L^{''}(t)\big\lvert=&\lvert\frac{d}{dt}\int_{0}^{2\pi}\varphi d\theta-L(t)\lambda(t)\lvert=
\bigg\lvert\int_{0}^{2\pi}\varphi_td\theta-L^{'}(t)\lambda(t)-L(t)\lambda^{'}(t)\bigg\lvert\\
=&\bigg\lvert\int_{0}^{2\pi}n\varphi^{1-\frac{1}{n}}[\varphi_{\theta\theta}+\varphi-\lambda(t)\varphi^{\frac{1}{n}}]d\theta
-\bigg[\int_{0}^{2\pi}\varphi d\theta-L(t)\lambda(t)\bigg]\lambda(t)\\
&-\frac{(n+1)L(t)}{2nA(0)}\int^{2\pi}_{0}\varphi^{\frac{1}{n}}
\bigg(n\varphi^{1-\frac{1}{n}}[\varphi_{\theta\theta}+\varphi-\lambda(t)\varphi^{\frac{1}{n}}]\bigg)d\theta\bigg\lvert\\
=&\bigg\lvert (1-n)\int^{2\pi}_{0}\varphi^{-\frac{1}{n}}(\varphi_{\theta})^{2}d\theta+n\int^{2\pi}_{0}\varphi^{2-\frac{1}{n}}d\theta
-(n+1)\lambda(t)\int^{2\pi}_{0}\varphi d\theta+L(t)\lambda^2(t)\\
&-\frac{(n+1)L(t)}{2A(0)}\bigg[\int^{2\pi}_{0}-(\varphi_{\theta})^{2}d\theta+\int^{2\pi}_{0}\varphi^{2}d\theta
-\frac{1}{2A(0)}\bigg(\int^{2\pi}_{0}\varphi^{1+\frac{1}{n}}d\theta\bigg)^{2}\bigg]\bigg\lvert
\end{align*}
Since $\varphi, \varphi_{\theta}, L(t)$ and $\lambda(t)$ are all bounded quantities, there exists a constant $C>0$ such that $\big\lvert L^{''}(t)\big\lvert\le C$. Thus we have $f(t)=-L^{'}(t)\to0$ as $t\to\infty$.
\end{proof}
Combining the Arzela-Ascoli theorem with Lemmas \ref{ntjj} and \ref{lto0} yields the following smooth convergence result.
\begin{lemma}\label{nuin} ($C^{\infty}$ convergence of $\omega$)
Under the flow \eqref{fl2}, we have
\begin{equation}
\lim\limits_{t\to\infty}\bigg\Vert\varphi(\theta,t)-\bigg(\sqrt{\frac{A(0)}{\pi}}\bigg)^{n}\bigg\Vert_{C^{i}(S^{1})}=0,
\quad \forall i=0,1,2,3\cdots.
\end{equation}
\end{lemma}
\begin{proof}
Since $\varphi$ and $\varphi_{\theta}$ are both uniformly bounded, for any sequence $t_i\to\infty$, there exists a subsequence (still denoted as $t_i$) such that $\lim\limits_{i\to\infty}\varphi(\theta,t_i)=\varphi_{\infty}(\theta)$ uniformly on $S^1$, where $\varphi_{\infty}(\theta)\ge0$ is some continuous bounded function on $S^1$. Set $\varphi_{\infty}(\theta)=\rho^n_{\infty}(\theta)$. By lemma \ref{lto0}, we have
\begin{equation*}
\frac{dL}{dt}=\int_{0}^{2\pi}\rho^n(\theta)d\theta-\frac{L}{2A}\int_{0}^{2\pi}\rho^{n+1}(\theta)d\theta\to0\quad\text{as}\quad t\to\infty.
\end{equation*}
The above implies
\begin{equation}\label{rinl}
\int_{0}^{2\pi}\rho^n_{\infty}(\theta)d\theta=\frac{1}{2A}
\int_{0}^{2\pi}\rho_{\infty}(\theta)d\theta\int_{0}^{2\pi}\rho^{n+1}_{\infty}(\theta)d\theta,
\end{equation}
where we have used the identity
\begin{equation*}
\lim\limits_{i\to\infty}L(t_{i})=\lim\limits_{i\to\infty}\int^{2\pi}_{0}
\rho(\theta,t_{i}) d\theta=\int^{2\pi}_{0}\rho_{\infty}(\theta)d\theta.
\end{equation*}
We now define the double integral
\begin{equation}\label{jin}
J\triangleq\frac{1}{2} \int_0^{2\pi} \int_0^{2\pi} \left( \rho_\infty(x) - \rho_\infty(y) \right) \left( \rho_\infty^n(x) - \rho_\infty^n(y) \right) dx dy.
\end{equation}
Note that the integrand in \eqref{jin} is a nonnegative continuous function on $S^1\times S^1$, hence it necessarily follows that $J\ge0$. On the other hand, expanding the integrand and integrating gives
\begin{equation*}
J = 2\pi \int_0^{2\pi} \rho_\infty^{n+1}  d\theta - \bigg( \int_0^{2\pi} \rho_\infty  d\theta \bigg) \bigg( \int_0^{2\pi} \rho_\infty^n  d\theta \bigg).
\end{equation*}
Combining this with \eqref{rinl} we obtain
\begin{align*}
J &= 2\pi \int_0^{2\pi} \rho_\infty^{n+1}  d\theta - \frac{1}{2A}\bigg( \int_0^{2\pi} \rho_\infty  d\theta \bigg)^2 \bigg( \int_0^{2\pi} \rho_\infty^{n+1}  d\theta \bigg)\notag\\
&\le2\pi \int_0^{2\pi} \rho_\infty^{n+1}  d\theta-\frac{4\pi A}{2A}\int_0^{2\pi} \rho_\infty^{n+1}  d\theta=0,
\end{align*}
where we have used the isoperimetric inequality $\big(\int_0^{2\pi}\rho_\infty  d\theta \big)^2\geq 4\pi A$. From this, we deduce that $J=0$. The nonnegative integrand vanishes almost everywhere, implying $\rho_\infty(x) =\rho_\infty(y)$ for all $x, y\in S^1$. Thus, $\rho_\infty(\theta)$ must be a constant function with $\rho_\infty(\theta)=\sqrt{A/\pi}$ (since the flow is area-preserving). Consequently $\varphi_\infty= \left( \sqrt{A/\pi} \right)^n$. As every sequence has a subsequence converging to this constant, we must have
$\varphi(\theta,t)\to\big(\sqrt{A/\pi}\big)^n$ uniformly, proving $C^0$ convergence.

For $i\ge1$, the uniform convergence of $\varphi$ to a constant and the uniform bounds on $\partial^{i}\varphi/\partial\theta^{i}$ ensure that $\partial^{i}\varphi/\partial\theta^{i}\to0$ uniformly on $S^1$ as $t\to\infty$, by repeated application of Arzela-Ascoli theorem. This completes the proof.
\end{proof}
As an application of flow \eqref{fl2}, we can get the following inequality.
\begin{theorem}
If $\gamma$ is a convex closed curve with length $L$, area $A$ and curvature $\kappa$ , then for $n>0$,
\begin{equation}
\label{ine1}\int_{0}^{2\pi} \frac{1}{\kappa^{n+1} } d\theta \geq 2\pi \left( \sqrt{\frac{A}{\pi}} \right)^{n+1},
\end{equation}
with equality if and only if the curve is a circle.
\end{theorem}
\begin{proof}
As $t \to \infty$, the curve converges to a circle of radius $r = \sqrt{A/\pi}$, with constant curvature $\kappa_{\infty} = \sqrt{\pi/A}$. Thus,
\[
\lambda_c = \lim_{t\to\infty} \lambda(t) = \frac{1}{2A} \int_{0}^{2\pi} \frac{1}{\kappa_{\infty}^{n+1}}  d\theta = \pi^{\frac{1-n}{2}} A^{\frac{n-1}{2}}.
\]
Given that $\lambda(t)$ is decreasing, \eqref{ine1} can be easily derived.
The equality holds if and only if $\lambda(t) = \lambda_c$ for all $t$. Since $\lambda(t)$ is strictly decreasing unless stationary, this occurs if and only if the curve is stationary. The parabolicity of the flow ensures that circles are the unique stationary solutions. At any circle, $\kappa$ is constant, and direct computation shows equality in the inequality.
\end{proof}
\textbf{Acknowledgements}
This work is supported by Key Natural Science Project of Anhui Province (No. 2024AH050573).

\textbf{Data availability}
Data sharing  not applicable to this article as no datasets were generated or analysed during the current study.


\end{document}